\documentclass[11pt, a4paper]{article}
\usepackage{amsmath,bm,amssymb,latexsym,color,enumerate,extarrows}
\usepackage{authblk}
\usepackage[mathscr]{eucal}
\usepackage[colorlinks,
            linkcolor=blue,
            anchorcolor=green,
            citecolor=magenta
           ]{hyperref}
\newtheorem{thm}{Theorem}[section]

\newtheorem{lem}[thm]{Lemma}

\newtheorem{con}{Construction}

\newtheorem{fact}{Fact}
\newtheorem{rema}[thm]{Remark}

\newtheorem{example}{Example}[section]
\newtheorem{defin}{Definition}[section]

\newcommand{\proof}{{\it Proof.\quad}}
\newcommand{\qed}{\hfill\Box\medskip}

\allowdisplaybreaks
\usepackage{CJK}

\begin{document}

\renewcommand{\baselinestretch}{1.3}

\title{\bf Extremal cross $t$-intersecting families under $t$-covering number constraints for vector spaces}

\author[1]{Yu Zhu\thanks{E-mail: \texttt{zhu\_y@mail.bnu.edu.cn}}}
\author[1]{Benjian Lv\thanks{Corresponding author. E-mail: \texttt{bjlv@bnu.edu.cn}}}
\author[1]{Kaishun Wang\thanks{E-mail: \texttt{wangks@bnu.edu.cn}}}

\affil[1]{\small Laboratory of Mathematics and Complex Systems (Ministry of Education), School of Mathematical Sciences, Beijing Normal University, Beijing 100875, China}

\date{}
\maketitle
\begin{abstract}
Let $V$ be an $n$-dimensional vector space over the finite field $\mathbb{F}_q$, and ${V\brack k}$ denote the family of all $k$-dimensional subspaces of $V$. The families $\mathcal{F}\subseteq {V\brack k}$ and $\mathcal{G}\subseteq {V\brack \ell}$ are said to be cross $t$-intersecting if $\dim(F\cap G)\geq t$ for all $F\in\mathcal{F}$ and $G\in \mathcal{G}$. In this paper, we determine the extremal structures when $|\mathcal{F}||\mathcal{G}|$ attains the maximum value under the conditions $\dim\left(\cap_{F\in \mathcal{F}}F\right)<t$ and $\dim\left(\cap_{G\in \mathcal{G}}G\right)<t$.

\medskip



\end{abstract}

\section{Introduction}

Let $n$ and $k$ be positive integers, and $V$ an $n$-dimensional vector space over the finite field $\mathbb{F}_q$, where $q$ is a prime power.
Given subspaces $S$ and $M$ with $S \subseteq M \subseteq V$, denote by $[S, M]_k$ the family of all $k$-dimensional subspaces of $M$ containing $S$. In particular, write $[S\rangle_k=\big[S, V\big]_k$ and ${M\brack k}=[0, M]_k$. In the sequel, we will abbreviate `$k$-dimensional subspace'
as `$k$-subspace'. Recall that the cardinality of ${V\brack k}$ is the Gaussian binomial coefficient
$${n\brack k} = \prod_{i=0}^{k-1} \frac{q^{n-i}-1} {q^{k-i}-1}.$$
In addition, we set ${n\brack 0}=1$ and ${n\brack m}=0$ if $m$ is a negative integer, and write $\theta_n={n\brack 1}$.

Let $\mathcal{F}$ be a family of some subspaces of $V$. A subspace $T\subseteq V$ is called a $t$-\emph{cover} of $\mathcal{F}$ if $\dim(T\cap F)\geq t$ for all $F\in\mathcal{F}$, and the $t$-\emph{covering number} $\tau_t(\mathcal{F})$ of $\mathcal{F}$ is the minimum dimension of a $t$-cover of $\mathcal{F}$.
It is immediate that $\tau_t(\mathcal{F})\geq t$, and equality holds if and only if $\dim\left(\cap_{F\in\mathcal{F}}F\right)\geq t$.

For positive integers $k$, $\ell$ and $t$, two families
$\mathcal{F}\subseteq {V\brack k}$ and $\mathcal{G}\subseteq {V\brack \ell}$ are called \emph{cross} $t$-\emph{intersecting} if
$\dim(F\cap G)\geq t$ for every $F\in\mathcal{F}$ and $G\in \mathcal{G}$. They are said to be \emph{maximal} if $\mathcal{F}=\mathcal{F}'$ and $\mathcal{G}=\mathcal{G}'$ for every pair of cross $t$-intersecting families $\mathcal{F}'\subseteq {V\brack k}$ and $\mathcal{G}'\subseteq {V\brack \ell}$ satisfying $\mathcal{F}\subseteq\mathcal{F}'$ and $\mathcal{G}\subseteq\mathcal{G}'$.
Clearly,
$t\leq \tau_t(\mathcal{F})\leq \ell$ and $t\leq \tau_t(\mathcal{G})\leq k$.
In particular, if $\mathcal{F}=\mathcal{G}$, then $\mathcal{F}$ is called \emph{$t$-intersecting}.
We may mention that, the definition of cross $t$-intersecting for vector spaces is a natural extension of that for finite sets. We refer to \cite{Borg-2016, Cao-Lu-Lv-Wang-2024, Frankl-Wang-2022, He-Li-Wu-Zhang-2026, Pyber-1986, Tanaka-Tokushige-2025, Tokushige-2013, Zhang-Wu-2025, Zhu-Lv-Wang-2026} for extremal cross $t$-intersecting families under $t$-covering number constraints for finite sets.

The problem of maximizing the product $|\mathcal{F}||\mathcal{G}|$ is a natural product analogue of the Erd\H{o}s-Ko-Rado problem.
Tokushige \cite{Tokushige-2013} used the eigenvalue method to describe the structure of $\mathcal{F}$ and $\mathcal{G}$ maximizing $|\mathcal{F}||\mathcal{G}|$ for $k=\ell$. Suda and Tanaka \cite{Suda-Tanaka-2014} used semidefinite programming to settle the case of arbitrary $k$ and $\ell$ when $t=1$.
Recently, Wen and Lv \cite{Wen-Lv-2026} used the $t$-cover to resolve the case of arbitrary $k$ and $\ell$.
These results show that, for large $n$, $|\mathcal{F}||\mathcal{G}|\leq {n-t\brack k-t}{n-t\brack \ell-t}$ and equality implies $\tau_t(\mathcal{F}\cup \mathcal{G})=t$.
Cao et al. \cite{Cao-Lu-Lv-Wang-2023} determined all extremal configurations of cross $t$-intersecting families of arbitrary $k$ and $\ell$ satisfying $\tau_t\left(\mathcal{F}\cup\mathcal{G}\right)\geq t+1$. It follows that $\tau_t(\mathcal{F})=t$ or $\tau_t(\mathcal{G})=t$ for every such extremal configuration $(\mathcal{F},\mathcal{G})$.
This motivates the study of cross $t$-intersecting families satisfying $\tau_t(\mathcal{F})\geq t+1$ and $\tau_t(\mathcal{G})\geq t+1$. Let $\mathcal{F}\subseteq {V\brack k}$ and $\mathcal{G}\subseteq {V\brack \ell}$ be such families with $|\mathcal{F}||\mathcal{G}|$ attaining the maximum value. Recently, Yao et al. \cite{Yao-Liu-Wang-2025} determined the structures of $\mathcal{F}$ and $\mathcal{G}$ for $k=\ell$,
In this paper, we focus on the structures of $\mathcal{F}$ and $\mathcal{G}$ for arbitrary $k$ and $\ell$.
We first introduce the following constructions.

\begin{con}\label{C1}
$\mathcal{A}(k,t;Z)=\left\{F\in {V\brack k} : \dim(F\cap Z)\geq t+1\right\}$, where $Z\in {V\brack t+2}$.
\end{con}

\begin{con}\label{C4}
Let $X\in {V\brack k+1}$, $Y\in {V\brack \ell+1}$ and $T\in {V\brack t}$ such that $\dim(X\cap Y)\geq t+2$ and $T\subseteq X\cap Y$. Define
\small{$$\mathcal{H}(k,t;T,X,Y)=\left\{F\in {V\brack k} : T\subseteq F,~\dim(F\cap Y)\geq t+1\right\}\cup \left\{F\in {X\brack k} : \dim(F\cap T)=t-1\right\}.$$}
\end{con}

\begin{con}\label{C3}
Let $M\in {V\brack \ell+1}$, and $L\in{M\brack t+1}$. Define
\begin{align*}
\mathcal{C}_1(\ell,t;M,L)&=\left\{G\in {V\brack \ell} : L\subseteq G\right\}\cup\left\{G\in {M\brack \ell} : \dim(G\cap L)=t\right\},\\
\mathcal{C}_2(k,t;M,L)&=\left\{F\in {V\brack k} : L\subseteq F\right\}\cup\left\{F\in {V\brack k} : \dim(F\cap L)=t,~\dim(F\cap M)\geq t+1\right\}.
\end{align*}
\end{con}

Obviously, each of the pairs $\left(\mathcal{A}(k,t;Z), ~\mathcal{A}(\ell,t;Z)\right)$, $\left(\mathcal{H}(k,t;T,X,Y), ~\mathcal{H}(\ell,t;T,Y,X)\right)$, and $\left(\mathcal{C}_2(k,t;M,L),~\mathcal{C}_1(\ell,t;M,L)\right)$ is cross $t$-intersecting if $\min\{k,\ell\}\geq t+2$. Our main result characterizes the extremal structures of cross $t$-intersecting families, whose $t$-covering numbers are at least $t+1$.

\begin{thm}\label{main-1}
Let $n$, $k_1$, $k_2$ and $t$ be positive integers with $k_1\geq k_2\geq t+2$ and $n\geq 3k_1+k_2-t+5$. Suppose that $\mathcal{F}_1\subseteq {V\brack k_1}$ and $\mathcal{F}_2\subseteq {V\brack k_2}$ are cross $t$-intersecting families with $\tau_t(\mathcal{F}_1)\geq t+1$ and $\tau_t(\mathcal{F}_2)\geq t+1$. If $|\mathcal{F}_1||\mathcal{F}_2|$ attains the maximum value, then $(\mathcal{F}_1,~\mathcal{F}_2)$ is one of the following pairs:
\begin{enumerate}[{\rm(i)}]
  \item $\left(\mathcal{A}(k_1,t;Z), ~\mathcal{A}(k_2,t;Z)\right)$, where $Z\in {V\brack t+2}$,
  \item $\left(\mathcal{H}(k_1,t;T,X,Y), ~\mathcal{H}(k_2,t;T,Y,X)\right)$, where $X\in {V\brack k_1+1}$, $Y\in {V\brack k_2+1}$ and $T\in {V\brack t}$ such that $\dim(X\cap Y)\geq t+2$ and $T\subseteq X\cap Y$,
  \item $\left(\mathcal{C}_1(k_1,t;M,L), ~\mathcal{C}_2(k_2,t;M,L)\right)$, where $M\in {V\brack k_1+1}$, and $L\in{M\brack t+1}$.
\end{enumerate}
\end{thm}

The rest of this paper is organized as follows.
In Section \ref{2}, we study cross
$t$-intersecting families $\mathcal{F}$ and $\mathcal{G}$ with $\tau_t(\mathcal{F})=\tau_t(\mathcal{G})=t+1$, characterizing their structures or estimating the product of their sizes. In Section \ref{2.2}, we prove Theorem \ref{main-1}. In Section \ref{3}, we establish several inequalities used in this paper.

\section{$\tau_t(\mathcal{F})=\tau_t(\mathcal{G})=t+1$} \label{2}

In this section, we first recall several auxiliary lemmas and facts that will be used throughout the paper. We then investigate the structures and cardinalities of the families $\mathcal{F}$ and $\mathcal{G}$ satisfying $\tau_t(\mathcal{F})=\tau_t(\mathcal{G})=t+1$.

\begin{lem}{\rm(\cite[Lemma~2.1]{Cao-Lu-Lv-Wang-2023})} \label{q<q}
Let $m$ and $i$ be positive integers with $i\leq m$. Then the following hold.
\begin{enumerate}[{\rm(i)}]
  \item ${m\brack i}=q^{i}{m-1\brack i}+{m-1\brack i-1}$ and ${m\brack i}=\frac{q^m-1}{q^i-1}\cdot {m-1\brack i-1}$.
  \item $q^{m-i}<\frac{q^m-1}{q^i-1}<q^{m-i+1}$ and $q^{i-m-1}<\frac{q^i-1}{q^m-1}<q^{i-m}$ if $i<m$.
  \item $q^{i(m-i)}\leq {m\brack i}<q^{i(m-i+1)}$, and $q^{i(m-i)}<{m\brack i}$ if $i<m$.
\end{enumerate}
\end{lem}

Let $e$, $r$, $m$, $h$, $m_1$ and $h_1$ be non-negative integers.
Suppose that $W$ is an $(e+r)$-dimensional vector space over $\mathbb{F}_q$ and $L$ is a fixed $r$-subspace of $W$. We say that an $m$-subspace $U$ of $W$ is of type $(m,h)$ if $\dim(U\cap L) = h$.
Define $\mathcal{M}(m,h;e+r,e)$ to be the set of all subspaces of $W$ of type $(m,h)$. Define
$N'(m_1, h_1;m, h;e+r, e)$ to be the number of subspaces of $W$ with type $(m, h)$ containing
a given subspace of type $(m_1, h_1)$. Clearly, $|\mathcal{M}(m,h;e+r,e)|=N'(0, 0;m, h;e+r, e)$.

\begin{lem}{\rm(\cite[Lemma~2.3]{Wang-Guo-Li-2011})} \label{|N'|}
$N'(m_1, h_1;m, h;e+r, e)\neq 0$ if and only if $0\leq h_1\leq h\leq r$ and $0\leq m_1-h_1\leq m-h\leq e$. Moreover, if $N'(m_1, h_1;m, h;e+r, e)\neq 0$, then
$$N'(m_1,h_1;m,h;e+r,e)=q^{(r-h)(m-h-m_1+h_1)}{e-(m_1-h_1)\brack (m-h)-(m_1-h_1)}{r-h_1\brack h-h_1}.$$
\end{lem}

Let $n$, $x$, $y$, $m$ and $t$ be positive integers. Write
\begin{align}
g(m,x,y,t)&=m{n-t-1\brack x-t-1}+\theta_{t+1}\theta_{y-t+1}^2{n-t-2\brack x-t-2},\label{equ-3}\\
a(x,t)&=q^{x-t-1}\theta_{t+2}{n-t-2\brack x-t-1}+{n-t-2\brack x-t-2}, \label{equ-1}\\
h(x,y,t)&={n-t\brack x-t}-q^{(x-t)(y-t+1)}{n-y-1\brack x-t}+q^{x-t+1}\theta_{t}, \label{equ-2}\\
c_2(x,y,t)&=q^{x-t}\theta_{t+1}\left({n-t-1\brack x-t}-q^{(x-t)(y-t)}{n-y-1\brack x-t}\right)+{n-t-1\brack x-t-1},\label{equ-11}\\
c_1(y,t)&={n-t-1\brack y-t-1}+q^{y-t}\theta_{t+1}.\label{equ-12}
\end{align}

\begin{rema}
It is clear that $g(m,x,y,t)$ is increasing in $m$. By applying Lemma~\ref{|N'|}, we obtain $|\mathcal{A}(k,t; Z)|=a(k,t)$, $|\mathcal{H}(k,t;T,X,Y)| = h(k,\ell,t)$, $|\mathcal{C}_1(\ell,t;M,L)|=c_1(\ell,t)$ and $|\mathcal{C}_2(k,t;M,L)|=c_2(k,\ell,t)$ for Constructions \ref{C1}, \ref{C4} and \ref{C3}.
\end{rema}

In the remaining of this section, we always assume that $n$, $k$, $\ell$ and $t$ are positive integers satisfying $\min\{k, \ell\}\geq t+2$ and $n\geq 2\cdot\max\{k,\ell\}+k+\ell-t+5$.
Let $\mathcal{F} \subseteq {V\brack k}$ and $\mathcal{G} \subseteq {V\brack \ell}$ be maximal cross $t$-intersecting families with $\tau_t(\mathcal{F})=\tau_t(\mathcal{G})=t+1$. Set $\mathcal{T}_f$ and $\mathcal{T}_g$ denote the collections of all $(t+1)$-dimensional $t$-covers of $\mathcal{F}$ and $\mathcal{G}$, respectively.
Denote $M_f = \sum_{A \in \mathcal{T}_f} A$ and $M_g = \sum_{B \in \mathcal{T}_g} B$. By the maximality of $\mathcal{F}$ and $\mathcal{G}$, it is clear that
\begin{align}\label{abs}
\bigcup_{B\in \mathcal{T}_g}[B\rangle_k\subseteq \mathcal{F}~~~~\text{and}~~~~\bigcup_{A\in \mathcal{T}_f}[A\rangle_\ell\subseteq \mathcal{G}.
\end{align}

For a subspace $S$ of $V$, define $\mathcal{F}_S=\{F\in \mathcal{F} : S\subseteq F\}$.
The following facts will be used frequently.

\begin{fact}{\rm (\cite[Remark~(ii)~in~\S9.3]{BCN})}\label{DRG}
If $\mathcal{T}\subseteq {V\brack t+1}$ is a maximal t-intersecting family, then $\mathcal{T}\simeq {M\brack t+1}$ or $\mathcal{T}\simeq[T\rangle_{t+1}$, where $M\in {V\brack t+2}$ and $T\in {V\brack t}$.
\end{fact}

\begin{fact}{\rm(\cite[Lemma~2.7]{Cao-Lu-Lv-Wang-2023})} \label{s-tt}
  $\mathcal{T}_f$ and $\mathcal{T}_g$ are cross $t$-intersecting families.
\end{fact}

The following lemma provides an upper bound on $|\mathcal{F}||\mathcal{G}|$ and will be used frequently in the subsequent proofs.

\begin{lem}\label{not-contain}
We have $|\mathcal{F}|\leq g(|\mathcal{T}_g|,k,\ell,t)$ and $|\mathcal{G}|\leq g(|\mathcal{T}_f|,\ell,k,t)$.
\end{lem}
\proof By symmetry, we only prove that $|\mathcal{F}|\leq g(|\mathcal{T}_g|,k,\ell,t)$.
Let $\mathcal{F}^{\prime}$ be the collection of all members of $\mathcal{F}$ that do not contain any elements of $\mathcal{T}_g$.  It is obvious that
\begin{align*}
|\mathcal{F}|\leq |\mathcal{T}_g|{n-t-1\brack k-t-1}+ |\mathcal{F}^{\prime}|.
\end{align*}
By \cite[Proposition~2.3]{Liu-Wang-Yao-2026}, we have
$|\mathcal{F}^{\prime}|\leq \theta_{t+1}\theta_{\ell-t+1}^2{n-t-2\brack k-t-2},$
and so the lemma follows.$\qed$

As shown in Lemma \ref{not-contain}, $|\mathcal{T}_g|$ directly determines the leading term of the inequality. In particular, when $n$ is sufficiently large, the leading term dominates the bound in magnitude, and the second term becomes negligible by comparison. Thus, in the following, we will give the upper bounds of $|\mathcal{F}||\mathcal{G}|$ by estimating the sizes of $\mathcal{T}_f$ and $\mathcal{T}_g$.

We divide the discussion into the following cases: (i) both $\mathcal{T}_f$ and $\mathcal{T}_g$ are $t$-intersecting families; (ii) exactly one of them is $t$-intersecting; (iii) neither of them is $t$-intersecting.
By Fact \ref{s-tt}, both $\mathcal{T}_f$ and $\mathcal{T}_g$ are $t$-intersecting if and only if $\mathcal{T}_f\cup \mathcal{T}_g$ is $t$-intersecting.

\begin{lem}\label{m}
Suppose $\mathcal{T}_f\cup \mathcal{T}_g$ is $t$-intersecting with $\tau_t(\mathcal{T}_f\cup \mathcal{T}_g)=t$.
\begin{enumerate}[{\rm(i)}]
  \item We have $|\mathcal{T}_f|\leq \theta_{k-t+1}$ and $|\mathcal{T}_g|\leq \theta_{\ell-t+1 }$.
  \item Assume that $(|\mathcal{T}_f|,|\mathcal{T}_g|)=\left(\theta_{k-t+1}, \theta_{\ell-t+1}\right)$. Then $\dim(M_f\cap M_g)\geq t+1$, and the following statements hold.
\begin{enumerate}
  \item[{\rm(iia)}] If $\dim(M_f\cap M_g)\geq t+2$, then $(\mathcal{F},\mathcal{G})= (\mathcal{H}(k,t;T,X,Y),\mathcal{H}(\ell,t;T,Y,X))$ for some $X\in {V\brack k+1}$ and $Y\in {V\brack \ell+1}$ satisfy $\dim(X\cap Y)\geq t+2$ and $T\in {X\cap Y\brack t}$.
  \item[{\rm(iib)}] If $\dim(M_f\cap M_g)= t+1$, then $$|\mathcal{F}||\mathcal{G}| \leq\left(h(k,\ell,t)-q^{k-t}\left(q\theta_{t}-1\right)\right)\left(h(\ell,k,t)-q^{\ell-t} \left(q\theta_{t}-1\right)\right).$$
\end{enumerate}
\end{enumerate}
\end{lem}
\proof (i) Since $\tau_t(\mathcal{T}_f\cup \mathcal{T}_g)=t$, there exists a $t$-subspace $T$ of $V$ such that  $\mathcal{T}_f\cup \mathcal{T}_g\subseteq [T\rangle_{t+1}$.
Hence, $\mathcal{T}_f\subseteq [T,M_f]_{t+1}$, $\mathcal{T}_g\subseteq [T,M_g]_{t+1}$, and therefore
\begin{align}\label{T=M-t}
|\mathcal{T}_f|\leq \theta_{\dim M_f-t} ~~~~\text{and}~~~~|\mathcal{T}_g|\leq\theta_{\dim M_g-t}.
\end{align}

It follows from \eqref{abs} and $\tau_t(\mathcal{F})=\tau_t(\mathcal{G})=t+1$ that $\mathcal{F}_T$, $\mathcal{G}_T$, $\mathcal{F}\setminus \mathcal{F}_T$ and $\mathcal{G}\setminus \mathcal{G}_T$ are all nonempty. Let $F\in \mathcal{F}\setminus \mathcal{F}_T$ and $A\in \mathcal{T}_f$. Then $\dim(F\cap T)\leq t-1$ and $T\subseteq A$. Hence,
$$\dim((F\cap A)+T)=\dim(F\cap A)+\dim T-\dim(F\cap A\cap T)\geq t+1.$$
Since $(F\cap A)+T\subseteq A$ and $\dim A=t+1$, we obtain $A=(F\cap A)+T\subseteq F+T$.
It follows from $T\nsubseteq F$ and $T\subseteq A$ that $\dim(F\cap T)=t-1$.
Therefore, for all $F\in \mathcal{F}\setminus \mathcal{F}_T$, we have $M_f=\sum_{A\in \mathcal{T}_f}A\subseteq F+T$,
and so
$\dim M_f\leq \dim(F+ T)=k+1$. Similarly, for each $G\in \mathcal{G}\setminus \mathcal{G}_T$, we also have $\dim(G\cap T)=t-1$ and $M_g\subseteq G+T$, and so $\dim M_g\leq \ell+1$. Combining this with \eqref{T=M-t}, we obtain (i) follows.

(ii) Assume that $(|\mathcal{T}_f|, |\mathcal{T}_g|)=\left(\theta_{k-t+1},\theta_{\ell-t+1}\right)$. Let $F\in \mathcal{F}\setminus \mathcal{F}_T$ and $G\in \mathcal{G}\setminus \mathcal{G}_T$. By the proof of (i), we have
\begin{align}\label{M=FG+T}
\dim(F\cap T)=\dim(G\cap T)=t-1,~~~~
M_f=F+T~~~~\text{and}~~~~M_g=G+T.
\end{align}
Then $\dim M_f=k+1$, $\dim M_g=\ell+1$ and
$$\dim(M_f\cap M_g)=\dim((F+T)\cap (G+T))\geq \dim(T+(F\cap G))\geq t+1.$$

Set
\begin{align*}
\mathcal{F}_0=\left\{F\in {V\brack k} : T\subseteq F,~\dim(F\cap M_g)\geq t+1\right\},\\
\mathcal{G}_0=\left\{G\in {V\brack \ell} : T\subseteq G,~\dim(G\cap M_f)\geq t+1\right\}.
\end{align*}
For each $F'\in \mathcal{F}_T$, by \eqref{M=FG+T}, we have
$$\dim(F'\cap M_g)=\dim(F'\cap (G+T))\geq \dim((F'\cap G)+(F'\cap T))\geq t+1.$$
Hence, $F'\in\mathcal{F}_0$, and so $\mathcal{F}_T\subseteq \mathcal{F}_0$. Similarly, one can prove $\mathcal{G}_T\subseteq \mathcal{G}_0$. Therefore,
$$\mathcal{F}\subseteq \mathcal{F}_0\cup\left(\mathcal{F}\setminus \mathcal{F}_T\right)~~~~\text{and}~~~~\mathcal{G}\subseteq \mathcal{G}_0\cup\left(\mathcal{G}\setminus \mathcal{G}_T\right).$$

(iia) Suppose $\dim(M_f\cap M_g)\geq t+2$.
Set
$$\mathcal{F}_1=\left\{F\in {M_f\brack k} : \dim(F\cap T)=t-1\right\}~~\text{and}~~
\mathcal{G}_1= \left\{G\in {M_g\brack \ell} : \dim(G\cap T)=t-1\right\}.$$
It follows from \eqref{M=FG+T} that $\mathcal{F}\setminus \mathcal{F}_T\subseteq \mathcal{F}_1$ and $\mathcal{G}\setminus \mathcal{G}_T\subseteq \mathcal{G}_1$. Hence, $\mathcal{F}\subseteq \mathcal{F}_0\cup \mathcal{F}_1$ and
$\mathcal{G}\subseteq \mathcal{G}_0\cup\mathcal{G}_1.$
It is routine to check that $\mathcal{F}_0\cup \mathcal{F}_1$ and $\mathcal{G}_0\cup \mathcal{G}_1$ are cross $t$-intersecting. By the maximality of $\mathcal{F}$ and $\mathcal{G}$, we have
$\mathcal{F}=\mathcal{F}_0\cup \mathcal{F}_1$ and $\mathcal{G}=\mathcal{G}_0\cup \mathcal{G}_1$.
It follows directly that $T\in {M_f\cap M_g\brack t}$ and $$(\mathcal{F},~\mathcal{G})= (\mathcal{H}(k,t;T,M_f,M_g),~\mathcal{H}(\ell,t;T,M_g,M_f)).$$
Then (iia) holds.

(iib) Suppose $\dim(M_f\cap M_g)=t+1$. Let $S=M_f\cap M_g$, $F_1\in \mathcal{F}\setminus \mathcal{F}_T$ and $G_1\in \mathcal{G}\setminus \mathcal{G}_T$.
Then $\dim(F_1\cap S)\leq t$, $\dim(G_1\cap S)\leq t$, and $F_1\cap G_1\subsetneq S$ by \eqref{M=FG+T}. It follows from $\dim(F_1\cap G_1)\geq t$ that $\dim(F_1\cap G_1\cap S)=t$.
Hence, $F_1\cap G_1\cap S=F_1\cap S=G_1\cap S$, which is a fixed $t$-subspace of $S$.
Therefore,
\begin{align*}
\mathcal{F}\subseteq\mathcal{F}_0\cup \left\{F\in {M_f\brack k} : F\cap S=F_1\cap S\right\},~~~~
\mathcal{G}\subseteq\mathcal{G}_0\cup \left\{G\in {M_g\brack \ell} : G\cap S=F_1\cap S\right\}.
\end{align*}
By Lemma~\ref{|N'|}, we obtain $|\mathcal{F}|\leq h(k,\ell,t)-q^{k-t}\left(q\theta_{t}-1\right)$ and $|\mathcal{G}|\leq h(\ell,k,t)-q^{\ell-t}\left(q\theta_{t}-1\right)$. Then (iib) holds.$\qed$

\begin{lem}\label{equivalent}
  Suppose that $\mathcal{T}_f\cup\mathcal{T}_g$ is a $t$-intersecting family with $\tau_t(\mathcal{T}_f\cup\mathcal{T}_g)= t+1$. Then $\max\big\{|\mathcal{T}_f|, |\mathcal{T}_g|\big\}\leq \theta_{t+2}$. Moreover, when $|\mathcal{T}_f|=|\mathcal{T}_g|=\theta_{t+2}$, we have $(\mathcal{F},~\mathcal{G})=(\mathcal{A}(k,t; Z),~ \mathcal{A}(\ell,t;Z))$ for some $Z\in {V\brack t+2}$.
\end{lem}
\proof Since $\mathcal{T}_f\cup\mathcal{T}_g$ is a $t$-intersecting family with $\tau_t(\mathcal{T}_f\cup\mathcal{T}_g)= t+1$, by Fact~\ref{DRG}, we can choose a $(t+2)$-subspace $M$ such that $\mathcal{T}_f\cup\mathcal{T}_g\subseteq {M\brack t+1}$. Then $\max\big\{|\mathcal{T}_f|, |\mathcal{T}_g|\big\}\leq \theta_{t+2}$.

Suppose $|\mathcal{T}_f|=\theta_{t+2}$. From $|{M\brack t+1}|=\theta_{t+2}$, it follows that $\mathcal{T}_f={M\brack t+1}$. Then $\dim(F\cap M)\geq t+1$ for each $F\in \mathcal{F}$, and so $\mathcal{F}\subseteq \mathcal{A}(k,t;M)$.
Similarly, we have $\mathcal{G}\subseteq\mathcal{A}(\ell,t;M)$ when $|\mathcal{T}_g|=\theta_{t+2}$.
By the maximality of $\mathcal{F}$ and $\mathcal{G}$, we have $(\mathcal{F},\mathcal{G})=(\mathcal{A}(k,t;M), \mathcal{A}(\ell,t;M))
$. Taking $Z=M$ completes the proof.$\qed$

In the following lemmas, we consider the case when $\mathcal{T}_f\cup\mathcal{T}_g$ is not $t$-intersecting. By Fact~\ref{s-tt}, at least one of $\mathcal{T}_f$ and $\mathcal{T}_g$ is not $t$-intersecting. Without loss of generality, we may assume that $\mathcal{T}_g$ is not $t$-intersecting. Then there exist $B_1,~B_2\in \mathcal{T}_g$ such that $\dim(B_1\cap B_2)\leq t-1$. Let $A\in \mathcal{T}_f$. It follows from Fact \ref{s-tt} and $\dim A=t+1$ that
\begin{align*}
t-1\geq \dim(A\cap B_1\cap B_2)&=\dim(A\cap B_1)+\dim(A\cap B_2)-\dim((A\cap B_1)+(A\cap B_2))\\
&\geq 2t-\dim(A\cap (B_1+ B_2))\geq t-1.
\end{align*}
Therefore,
\begin{align}\label{bbabb}
\dim(B_1\cap B_2)=t-1,~~\text{and}~~B_1\cap B_2\subseteq A\subseteq B_1+B_2~~\text{for~every}~~A\in \mathcal{T}_f.
\end{align}

\begin{lem}\label{T=1}
Suppose $\mathcal{T}_g$ is not $t$-intersecting, and $|\mathcal{T}_f|=1$. Then the following statements hold.
\begin{enumerate}[{\rm(i)}]
  \item If $\mathcal{T}_f\nsubseteq \mathcal{T}_g$, then $|\mathcal{T}_g|\leq \theta_{2}\theta_{\ell-t+1}$.
  \item If $\mathcal{T}_f\subseteq \mathcal{T}_g$, then $|\mathcal{T}_g|\leq q\theta_{t+1}\theta_{\ell-t}+1$. Moreover, when $|\mathcal{T}_g|=q\theta_{t+1}\theta_{\ell-t}+1$, we have $(\mathcal{F},\mathcal{G})= (\mathcal{C}_2(k,t;M,L),~\mathcal{C}_1(\ell,t;M,L))$ for some $M\in{V\brack \ell+1}$ and $L\in{M\brack t+1}$.
\end{enumerate}
\end{lem}
\proof (i) Suppose $\mathcal{T}_f=\{M_f\}\nsubseteq \mathcal{T}_g$. Then there exists $G'\in \mathcal{G}$ such that $\dim(G'\cap M_f)\leq t-1$. Let $B\in \mathcal{T}_g$. Since $\dim(G'\cap B)\geq t$ and $\dim(B\cap M_f)\geq t$, we have
\begin{align*}
t-1\geq\dim(G'\cap M_f)&\geq\dim((G'\cap B)\cap (M_f\cap B))
\geq t-1.
\end{align*}
Hence,
$\dim(G'\cap M_f)=t-1,~G'\cap M_f\subseteq B$, $\dim(G'\cap B)=\dim(M_f\cap B)=t$ and $G'\cap B\neq M_f\cap B$,
implying that $B=(B\cap G')+(B\cap M_f)\subseteq G'+M_f$. Therefore,
$$\mathcal{T}_g\subseteq \left\{B\in {G'+M_f\brack t+1} : G'\cap M_f\subseteq B,~\dim(B\cap M_f)=\dim(B\cap G')= t\right\},$$
and $|\mathcal{T}_g|\leq \theta_{2}\theta_{\ell-t+1}$ from Lemma~\ref{|N'|}, and (i) holds.

(ii) Since $\mathcal{T}_f=\{M_f\}\subseteq \mathcal{T}_g$, by Fact \ref{s-tt}, we have $\dim(B\cap M_f)=t$ for each $B\in \mathcal{T}_g\setminus \{M_f\}$, and thus
$$\mathcal{T}_g\setminus \{M_f\}=\bigcup_{T\in {M_f\brack t}}\big((\mathcal{T}_g)_T\setminus \{M_f\}\big).$$

Since $M_f\in \mathcal{T}_g$ and $\tau_t(\mathcal{G})=t+1$, there exist $G_1,G_2\in \mathcal{G}\setminus[M_f\rangle_{\ell}$ such that $G_1\cap M_f\neq G_2\cap M_f$.
For each $T\in {M_f\brack t}$, let $G_T\in \{G_1,G_2\}$ with $G_T\cap M_f\neq T$. Then $\dim(G_T\cap M_f)=t$ and $\dim(G_T\cap T)\leq t-1$. For each $B\in (\mathcal{T}_g)_T\setminus \{M_f\}$, we have $B\cap M_f=T$ and
\begin{align*}
t+1=\dim B&\geq \dim((B\cap G_T)+(B\cap M_f)) \geq 2t-\dim(G_T\cap T)\geq t+1,
\end{align*}
implying that $\dim(B\cap G_T)=t$ and  $B=(B\cap G_T)+(B\cap M_f)\subseteq G_T+M_f$.
Hence,
$$(\mathcal{T}_g)_T\setminus \{M_f\}\subseteq \left\{B\in {G_T+M_f\brack t+1} : \dim(B\cap G_T)=t,~B\cap M_f=T\right\},$$
and $|(\mathcal{T}_g)_T\setminus \{M_f\}|\leq N'(t,t;t+1,t;\ell+1,\ell-t)=q\theta_{\ell-t}$ from Lemma~\ref{|N'|}. Therefore,
$$|\mathcal{T}_g|=\sum_{T\in {M_f\brack t}} |(\mathcal{T}_g)_T\setminus \{M_f\}|+1\leq q\theta_{t+1}\theta_{\ell-t}+1.$$
This proves the former part of (ii).

Suppose $|\mathcal{T}_g|=q\theta_{t+1}\theta_{\ell-t}+1$.
Then for every $T\in {M_f\brack t}$, we have $|(\mathcal{T}_g)_T\setminus \{M_f\}|= q\theta_{\ell-t}$ and
$$M_f+\sum_{B\in (\mathcal{T}_g)_T\setminus \{M_f\}}B=G_T+M_f$$
for each $G_T\in \{G_1,G_2\}$ with $\dim(G_T\cap T)<t$.
Since $|{M_f\brack t}|\geq 3$, there exists $T_0\in {M_f\brack t}$ such that $\dim(G_1\cap T_0)<t$ and $\dim(G_2 \cap T_0)<t$.
Then all members of $\mathcal{T}_g$ are contained in either $G_1+M_f$ or $G_2+M_f$, and
$$G_1+M_f=M_f+\sum_{B\in (\mathcal{T}_g)_{T_0}\setminus \{M_f\}}B=G_2+M_f,$$
implying that $M_g=G_1+M_f=G_2+M_f$ is an $(\ell+1)$-subspace. Hence,
\begin{align}\label{Tg=TCx}
\mathcal{T}_g\setminus \{M_f\}=\bigcup_{T\in {M_f\brack t}}\left\{B\in {M_g\brack t+1} : B\cap M_f=T\right\}.
\end{align}

Set
\begin{align*}
\mathcal{F}'&=\left\{F\in {V\brack k} : \dim(F\cap M_f)=t, ~\dim(F\cap M_g)\geq t+1\right\}\bigcup [M_f\rangle_k,\\
\mathcal{G}'&=\left\{G\in {M_g\brack \ell} : \dim(G\cap M_f)=t\right\}\bigcup [M_f\rangle_\ell.
\end{align*}
We first show that $\mathcal{F}'$ and $\mathcal{G}'$ are cross $t$-intersecting. This is clear if either $F\in\mathcal{F}'$ or $G\in \mathcal{G}'$ contains $M_f$. For $F\in \mathcal{F}'\setminus [M_f\rangle_k$ and $G\in \mathcal{G}'\setminus [M_f\rangle_\ell$, we have
$$\dim(F\cap G)\geq \dim(F\cap G\cap M_g)\geq \dim(F\cap M_g)+\dim G-\dim M_g\geq t.$$
Thus $\mathcal{F}'$ and $\mathcal{G}'$ are cross $t$-intersecting.

By \eqref{abs}, we see that $[M_f\rangle_\ell\subseteq \mathcal{G}$ and $[M_f\rangle_k\subseteq \mathcal{F}$.
Let $G\in \mathcal{G}\setminus [M_f\rangle_\ell$. Then $\dim(G\cap M_f)=t$. There exists a $t$-subspace $T\subseteq M_f$ such that $G\cap M_f\neq T$, and so
$$M_f+\sum_{B\in (\mathcal{T}_g)_T\setminus \{M_f\}}B=G+M_f\subseteq M_g$$ from \eqref{Tg=TCx}, implying that $G\subseteq M_g$. Therefore, $G\in \mathcal{G}'$ and so $\mathcal{G}\subseteq \mathcal{G}'$. Let $F\in \mathcal{F}\setminus [M_f\rangle_k$. Then $\dim(F\cap M_f)=t$. Since $G_1\cap M_f\neq G_2\cap M_f$, we assume that $G_1\cap M_f\neq F\cap M_f$, and so $\dim(G_1\cap M_f\cap F)=t-1$. Using $M_g=G_1+M_f$, we obtain
\begin{align*}
\dim(F\cap M_g)\geq  \dim(F\cap G_1) +\dim(F\cap M_f) -\dim(G_1\cap M_f\cap F)\geq t+1,
\end{align*}
implying that $F\in \mathcal{F}'$. Therefore, $\mathcal{F}\subseteq \mathcal{F}'$. By maximality, we obtain $\mathcal{F}=\mathcal{F}'$ and $\mathcal{G}=\mathcal{G}'$. Taking $M=M_g$ and $L=M_f$ completes the proof.$\qed$

\begin{lem}\label{T=2-C}
Suppose that $\mathcal{T}_g$ is not $t$-intersecting, and $\mathcal{T}_f$ is a $t$-intersecting family with $|\mathcal{T}_f|\geq 2$. Then $|\mathcal{T}_f|=q+1$ and $|\mathcal{T}_g|\leq \max\left\{\theta_{\ell-t+1}+\theta_{t+2}-\theta_{2},~ \theta_{\ell-t+2}\right\}$.
\end{lem}
\proof Let $B_1$ and $B_2$ be distinct elements in $\mathcal{T}_g$ with $\dim(B_1\cap B_2)\leq t-1$. It follows from Fact \ref{s-tt} that $B_1,B_2\notin \mathcal{T}_f$. From \eqref{bbabb}, we know that $\dim(B_1+ B_2)=t+3$ and $M_f=\sum_{A\in \mathcal{T}_f}A\subseteq B_1+ B_2\subseteq M_g$.

Since $\mathcal{T}_f$ is $t$-intersecting with $|\mathcal{T}_f|\geq 2$, there exist distinct $A_1,A_2\in\mathcal{T}_f$ such that $\dim(A_1+A_2)=t+2$. Clearly, $\tau_t(\mathcal{T}_f)\in \{t,t+1\}$. Let $\mathcal{U}_i=\mathcal{T}_f\cup\{B_i\}$ for $i\in \{1,2\}$. By Fact \ref{s-tt}, $\mathcal{U}_i=\mathcal{T}_f\cup\{B_i\}$ is $t$-intersecting, and $\tau_t(\mathcal{T}_f)\leq \tau_t(\mathcal{U}_i)\leq t+1$.
If $\tau_t(\mathcal{T}_f)=t+1$, then $B_i\in \mathcal{U}_i\subseteq {A_1+A_2\brack t+1}$ from Fact \ref{DRG}, contradicting $\dim(B_1\cap B_2)\leq t-1$.
Hence, $\tau_t(\mathcal{T}_f)=t$, and so $\mathcal{T}_f\subseteq [A_1\cap A_2,M_f]_{t+1}$.
Suppose $\dim M_f> t+2$. Then there exists $A_3\in \mathcal{T}_f\setminus \{A_1,A_2\}$ such that $A_3\nsubseteq A_1+A_2$. Since $\mathcal{T}_f$ is $t$-intersecting, we know $A_3\cap A_1=A_3\cap A_2=A_1\cap A_2$.
For each $B\in \mathcal{T}_g$ and $j\in \{1,2,3\}$, it follows from $\dim(B\cap A_j)\geq t$ that $A_1\cap A_2\subseteq B$, contradicting that $\mathcal{T}_g$ is not $t$-intersecting.
Therefore, $M_f=A_1+A_2$. Let $T=A_1\cap A_2$. Then $\mathcal{T}_f\subseteq [T,M_f]_{t+1}$.

Let $F\in \mathcal{F}$ and $A\in [T,M_f]_{t+1}$. It is clear that $\dim(F\cap M_f)\geq t$. If $\dim(F\cap M_f)\geq t+1$, then $\dim(F\cap A)\geq t$ is immediate. Now we assume that $\dim(F\cap M_f)=t$. Since $\dim(F\cap A_1)\geq t$ and $\dim(F\cap A_2)\geq t$, we obtain $F\cap M_f=F\cap A_1=F\cap A_2\subseteq A_1\cap A_2\subseteq A$, implying that $\dim(F\cap A)\geq t$. Then $A\in \mathcal{T}_f$. Therefore, $\mathcal{T}_f=[T,M_f]_{t+1}$ and consequently $|\mathcal{T}_f|=q+1$. Then the former part holds.

For each $B\in \mathcal{T}_g$, since $\dim(B\cap A_1)\geq t$ and $\dim(B\cap A_2)\geq t$, it follows from Fact \ref{DRG} that either $T\subseteq B$ or $B\in {M_f\brack t+1}$. Then
\begin{align*}
\mathcal{T}_g\subseteq\left([T,M_g]_{t+1}\setminus \mathcal{T}_f\right)\cup \left(\mathcal{T}_g\cap{M_f\brack t+1}\right).
\end{align*}
Since $[T,M_g]_{t+1}\cap {M_f\brack t+1}=\mathcal{T}_f$, by Lemma~\ref{|N'|}, we obtain
\begin{align}\label{Tg-size}
|\mathcal{T}_g|\leq  \theta_{\dim M_g-t} -\theta_{2} +\left|\mathcal{T}_g\cap {M_f\brack t+1}\right|.
\end{align}
Furthermore, as $\mathcal{T}_g$ is not $t$-intersecting, both $\mathcal{T}_g\cap\left([T,M_g]_{t+1}\setminus \mathcal{T}_f\right)$ and ${M_f\brack t+1}\cap\mathcal{T}_g$ are non-empty.

By $\tau_t(\mathcal{G})=t+1$, we have $\mathcal{G}\setminus \mathcal{G}_T\neq \emptyset$.
Let $B\in \mathcal{T}_g\cap\left([T,M_g]_{t+1}\setminus \mathcal{T}_f\right)$ and $G\in \mathcal{G}\setminus \mathcal{G}_T$. Then $\dim(G\cap B)\geq t$, and so $\dim(G\cap T)=t-1$.
Moreover, $\dim B\geq \dim((G\cap B)+T)\geq t+1$, implying that $B=(G\cap B)+T\subseteq G +M_f$. Therefore,
\begin{align*}
M_g=M_f+\sum_{B\in\mathcal{T}_g\cap\left([T,M_g]_{t+1}\setminus \mathcal{T}_f\right)}B\subseteq G +M_f~\text{for~all}~G\in \mathcal{G}\setminus \mathcal{G}_T,
\end{align*}
and so $\dim M_g\leq \ell+t+2-\dim(G\cap M_f)$.

Suppose $\dim(G\cap M_f)=t+1$
for some $G\in \mathcal{G}\setminus \mathcal{G}_T$. Then $\dim M_g\leq \ell+1$, and so
$|\mathcal{T}_g|
\leq \theta_{\ell-t+1}+\theta_{t+2}-\theta_{2}$
by \eqref{Tg-size}.

Suppose $\dim(G\cap M_f)=t$
for all $G\in \mathcal{G}\setminus \mathcal{G}_T$. We have $\dim M_g\leq \ell+2$. If $\left|\mathcal{T}_g\cap {M_f\brack t+1}\right|=1$, then $|\mathcal{T}_g|\leq\theta_{\ell-t+2}-q<\theta_{\ell-t+2}$ by \eqref{Tg-size}.
Now we assume $\left|\mathcal{T}_g\cap {M_f\brack t+1}\right|\geq 2$.
Let $G\in \mathcal{G}\setminus \mathcal{G}_T$, and $C_1, C_2\in \mathcal{T}_g\cap {M_f\brack t+1}$ be distinct subspaces. For $i\in \{1,2\}$, since $t\leq\dim(G\cap C_i)\leq\dim(G\cap M_f)=t$, we have $G\cap M_f= G\cap C_i\subseteq C_i$, and so $G\cap M_f=C_1\cap C_2$.
Then
\begin{align*}
\mathcal{T}_g\cap {M_f\brack t+1}\subseteq [C_1\cap C_2,M_f]_{t+1}.
\end{align*}
Hence, $|\mathcal{T}_g|\leq \theta_{\ell-t+2}$ by \eqref{Tg-size}.
Therefore
$|\mathcal{T}_g|\leq \max\left\{\theta_{\ell-t+1} +\theta_{t+2}-\theta_{2},~\theta_{\ell-t+2}\right\}$. This completes the proof. $\qed$

In what follows, for a subspace $X$ of $V$, we denote by $\mathcal{V}_X$ the subspaces of $V$ containing
$X$, and by $F/X$ the quotient of $F\in \mathcal{V}_X$. Recall that the members of $\mathcal{V}_X$ are in one-to-one
correspondence with the subspaces of $V/X$. Moreover, for $F,F' \in\mathcal{V}_X$, it is evident that
$(F \cap F')/X=(F/X)\cap (F'/X)$, and $\dim((F \cap F')/X) =\dim(F \cap F')-\dim(X)$.

\begin{lem}\label{both-non-inter}
Suppose neither $\mathcal{T}_f$ nor $\mathcal{T}_g$ is $t$-intersecting. If $|\mathcal{T}_f|=2$, then $|\mathcal{T}_g|\leq (q+1)^2$. If $|\mathcal{T}_f|\geq 3$, then $|\mathcal{T}_g|\leq 2q+1$.
\end{lem}
\proof
Since neither $\mathcal{T}_f$ nor $\mathcal{T}_g$ is $t$-intersecting, we have $|\mathcal{T}_f|\geq 2$ and $|\mathcal{T}_g|\geq 2$.
Let $A_1$ and $A_2$ be distinct elements in $\mathcal{T}_f$ with $\dim(A_1\cap A_2)\leq t-1$, and $B_1$ and $B_2$ be distinct elements in $\mathcal{T}_g$ with $\dim(B_1\cap B_2)\leq t-1$.
By \eqref{bbabb}, we have $\dim(A_1\cap A_2)=\dim(B_1\cap B_2)=t-1$, and $A_1\cap A_2\subseteq B\subseteq A_1+ A_2$ and $B_1\cap B_2\subseteq A \subseteq B_1+B_2$ for each $A\in \mathcal{T}_f$ and $B\in \mathcal{T}_g$, implying that $A_1\cap A_2=B_1\cap B_2$ and $M_g=A_1+ A_2=B_1+ B_2=M_f\in {V\brack t+3}$.

Let $\bar{M}=M_f/T=M_g/T$, $T=A_1\cap A_2=B_1\cap B_2$, $\mathcal{T}'_f=\left\{A/T : A\in \mathcal{T}_f\right\}$ and $\mathcal{T}'_g=\left\{B/T : B\in \mathcal{T}_g\right\}$. Then $\dim \bar{M}=4$, $|\mathcal{T}_f|=|\mathcal{T}'_f|$ and $|\mathcal{T}_g| = |\mathcal{T}'_g|$.
Set
\begin{align*}
\mathcal{A}&=\left\{\bar{A}\in {\bar{M}\brack 2} : \dim(\bar{A}\cap (B_1/T))\geq 1, \dim(\bar{A}\cap (B_2/T))\geq 1\right\},\\
\mathcal{B}&=\left\{\bar{B}\in {\bar{M}\brack 2} : \dim(\bar{B}\cap (A_1/T))\geq 1, \dim(\bar{B}\cap (A_2/T))\geq 1\right\}.
\end{align*}
Then $\mathcal{T}'_f \subseteq \mathcal{A}$ and $\mathcal{T}'_g \subseteq \mathcal{B}$. Moreover, $\mathcal{T}'_f$ and $\mathcal{T}'_g$ are cross $1$-intersecting, and neither of them is $1$-intersecting.

Suppose $|\mathcal{T}'_f|=2$. Then $|\mathcal{T}'_g|\leq|\mathcal{B}|=(q+1)^2$ from Lemma~\ref{|N'|}, and the former part of the lemma follows.

Suppose $|\mathcal{T}'_f|\geq 3$.
Let $\bar{C}\in \mathcal{T}'_f\setminus \{A_1/T,A_2/T\}$. Choose a basis $\bar{e}_1$, $\bar{e}_2$, $\bar{e}_3$, $\bar{e}_4$ of $\bar{M}$ such that
$$A_1/T=\langle\bar{e}_1,\bar{e}_2\rangle,~ A_2/T=\langle\bar{e}_3,\bar{e}_4\rangle,~ B_1/T=\langle\bar{e}_1,\bar{e}_3\rangle,~ B_2/T=\langle\bar{e}_2,\bar{e}_4\rangle.$$
By Fact \ref{s-tt}, we assume $\bar{C}=\langle\bar{u}_1,\bar{u}_2\rangle$ such that $\bar{u}_1=\alpha_1\bar{e}_1+\alpha_3\bar{e}_3,~ \bar{u}_2=\alpha_2\bar{e}_2+\alpha_4\bar{e}_4$, where $\alpha_1, \alpha_2, \alpha_3, \alpha_4\in \mathbb{F}_q$ and $(0,0)\notin \{(\alpha_1,\alpha_2), (\alpha_1,\alpha_3),(\alpha_2,\alpha_4) ,(\alpha_3,\alpha_4)\}$. Let $\mathcal{B}(\bar{C})=\left\{\bar{D}\in \mathcal{B} : \dim(\bar{D}\cap \bar{C})\neq 0\right\}$ and  $\bar{D}=\langle\bar{v}_1,\bar{v}_2\rangle\in \mathcal{B}(\bar{C})$. We assume $\bar{v}_1=\beta_{1}\bar{e}_1+\beta_{2}\bar{e}_2,~ \bar{v}_2=\beta_{3}\bar{e}_3+\beta_{4}\bar{e}_4$, where
$\beta_{1}, \beta_{2}, \beta_{3}, \beta_{4}\in \mathbb{F}_q$ with $(0,0)\notin\{(\beta_{1},\beta_{2}), (\beta_{3},\beta_{4})\}$. Since $\dim(\bar{D}\cap \bar{C})\neq 0$, the vectors $\bar{u}_1,\bar{u}_2,\bar{v}_1,\bar{v}_2$ are linearly dependent, and so
\begin{align*}
\begin{vmatrix}
\alpha_1 & 0 & \beta_{1} & 0\\
0 & \alpha_2 & \beta_{2} & 0\\
\alpha_3 & 0 & 0 & \beta_{3}\\
0 & \alpha_4 & 0 & \beta_{4}
\end{vmatrix}
=\alpha_1\alpha_4\beta_{2}\beta_{3} -\alpha_2\alpha_3\beta_{1}\beta_{4} =0.
\end{align*}

Since $(\alpha_1,\alpha_3)\neq (0,0) $ and $(\alpha_2,\alpha_4)\neq (0,0)$, we assume $\alpha_3\neq 0$ and $\alpha_4\neq 0$. (The remaining cases can be handled similarly.) After suitable rescaling, we may further normalize $\alpha_3=\alpha_4=1$.
Let $$N(\alpha_1, \alpha_2)=\left |\left\{(\beta_{1},\beta_{2}, \beta_{3},\beta_{4})\in \mathbb{F}_q^4 : \alpha_1\beta_{2}\beta_{3} =\alpha_2\beta_{1}\beta_{4},~(0,0)\notin\{(\beta_{1},\beta_{2}), (\beta_{3},\beta_{4})\}\right\}\right |.$$
Then $|\mathcal{B}(\bar{C})|\leq (q-1)^{-2}N(\alpha_1, \alpha_2)$.

By $(\alpha_1,\alpha_2)\neq (0,0)$, without loss of generality, we assume $\alpha_2\neq 0$.
If $\alpha_1=0$, then
$$N(\alpha_1, \alpha_2)=\left |\left\{(\beta_{1},\beta_{2}, \beta_{3},\beta_{4})\in \mathbb{F}_q^4 : \beta_{1}\beta_{4}=0,~(0,0)\notin\{(\beta_{1},\beta_{2}), (\beta_{3},\beta_{4})\}\right\}\right |=(2q+1)(q-1)^2.$$
If $\alpha_1\neq 0$, then $N(\alpha_1, \alpha_2)$ equals
$$\left |\left\{(\beta_{1},\beta_{2}, \beta_{3},\beta_{4})\in \mathbb{F}_q^4 : (\alpha_2^{-1}\alpha_1)\beta_{2}\beta_{3}=\beta_{1}\beta_{4}, ~(0,0)\notin\{(\beta_{1},\beta_{2}), (\beta_{3},\beta_{4})\}\right\}\right |=(q+2)(q-1)^2.$$
Hence, $|\mathcal{T}_g|=|\mathcal{T}'_g| \leq|\mathcal{B}(\bar{C})|\leq 2q+1$, and the latter part of the lemma follows.$\qed$

\section{Proof of the main theorem}\label{2.2}

Before proving Theorem \ref{main-1}, we will give an upper bound for the product $|\mathcal{F}||\mathcal{G}|$ when
$(\tau_t(\mathcal{F}),\tau_t(\mathcal{G}))\neq (t+1,t+1)$.

\begin{lem}\label{cfcg-t+2}
Let $n$, $k$, $\ell$ and $t$ be positive integers satisfying $\min\{k, \ell\}\geq t+2$ and $n\geq 2\cdot\max\{k,\ell\}+k+\ell-t+5$.
Suppose $\mathcal{F} \subseteq {V\brack k}$ and $\mathcal{G} \subseteq {V\brack \ell}$ are maximal cross $t$-intersecting families with
$(\tau_t(\mathcal{F}),\tau_t(\mathcal{G}))\neq (t+1,t+1)$. Then
$|\mathcal{F}||\mathcal{G}|<a(k,t)a(\ell,t).$
\end{lem}
\proof Applying \cite[Lemma~2.5]{Cao-Lu-Lv-Wang-2023} to $\mathcal{F}$ and $\mathcal{G}$, respectively, we have
\begin{align}\label{s-upper-F1}
	|\mathcal{F}|\leq\left\{
	\begin{array}{ll}
		\theta_{\ell-t+1}{\tau_t(\mathcal{F})\brack t}{n-t-1\brack k-t-1}, &\text{if}~\tau_t(\mathcal{G})=t+1,\vspace{0.1cm}\\
\theta_{\ell}^{\tau_t(\mathcal{G})-t-2}\theta_{\ell-t+1}^2 {\tau_t(\mathcal{F})\brack t}{n-\tau_t(\mathcal{G})\brack k-\tau_t(\mathcal{G})},&\text{if}~\tau_t(\mathcal{G})\geq t+2,
	\end{array}
	\right.
\end{align}
and
\begin{align}\label{s-upper-F2}
	|\mathcal{G}|\leq\left\{
	\begin{array}{ll}
		\theta_{k-t+1}{\tau_t(\mathcal{G})\brack t}{n-t-1\brack \ell-t-1}, &\text{if}~ \tau_t(\mathcal{F})=t+1, \vspace{0.1cm}\\
\theta_{k}^{\tau_t(\mathcal{F})-t-2}\theta_{k-t+1}^2
       {\tau_t(\mathcal{G})\brack t}{n-\tau_t(\mathcal{F})\brack\ell-\tau_t(\mathcal{F})},
       &\text{if}~\tau_t(\mathcal{F})\geq t+2.
	\end{array}
	\right.
\end{align}

Suppose $\tau_t(\mathcal{G})\geq t+2$  and $\tau_t(\mathcal{F})\geq t+2$. From (\ref{s-upper-F1}), (\ref{s-upper-F2}), \cite[Lemma~2.3~(ii)]{Cao-Lu-Lv-Wang-2023} and Lemma~\ref{q<q}, we have
\begin{align*}
|\mathcal{F}||\mathcal{G}|& \leq \theta_{\ell-t+1}^2\theta_{k-t+1}^2{t+2\brack 2}^2 {n-t-2\brack k-t-2} {n-t-2\brack\ell-t-2}\\
&=\theta_{\ell-t+1}^2\theta_{k-t+1}^2{t+2\brack 2}^2\cdot
\frac{(q^{k-t-1}-1)(q^{\ell-t-1}-1)} {(q^{n-k}-1)(q^{n-\ell}-1)}\cdot{n-t-2\brack k-t-1} {n-t-2\brack\ell-t-1} \\
&<q^{-2n+4k+4\ell-2t+6}\cdot{n-t-2\brack k-t-1} {n-t-2\brack\ell-t-1}
<a(k,t)a(\ell,t).
\end{align*}

Suppose $\tau_t(\mathcal{G})\geq t+2$ and $\tau_t(\mathcal{F})=t+1$. From (\ref{s-upper-F1}), (\ref{s-upper-F2}), \cite[Lemma~2.3~(ii)]{Cao-Lu-Lv-Wang-2023} and Lemma~\ref{q<q}, we have
\begin{align*}
|\mathcal{F}||\mathcal{G}|&\leq \theta_{t+1}\theta_{k-t+1}\theta_{\ell-t+1}^2 {t+2\brack 2}{n-t-2\brack k-t-2}{n-t-1\brack \ell-t-1}\\
&<q^{-n+3k+3\ell-2t+5} {n-t-2\brack k-t-1}{n-t-2\brack \ell-t-1}\\
&<q^{k+\ell-2t-2}\theta_{t+2}^2\cdot {n-t-2\brack k-t-1}{n-t-2\brack \ell-t-1}
<a(k,t)a(\ell,t).
\end{align*}
By symmetry, if $\tau_t(\mathcal{F})\geq t+2$ and $\tau_t(\mathcal{G})=t+1$, then $|\mathcal{F}||\mathcal{G}|<a(k,t)a(\ell,t)$ holds, and so this lemma follows.$\qed$

\textbf{\emph{Proof of Theorem \ref{main-1}.}}
Recall that the product of sizes of the families in Theorem \ref{main-1} (i), (ii) and (iii) are $a(k_1,t)a(k_2,t)$, $h(k_1,k_2,t)h(k_2,k_1,t)$ and  $c_1(k_1,t)c_2(k_2,k_1,t)$, respectively.
Let $\mathcal{F}_1\subseteq {V\brack k_1}$ and $\mathcal{F}_2\subseteq {V\brack k_2}$ be cross $t$-intersecting families with $\tau_t(\mathcal{F}_1) \geq t+1$ and $\tau_t(\mathcal{F}_2) \geq t+1$ such that $|\mathcal{F}_1||\mathcal{F}_2|$ is maximized.
By Lemma~\ref{cfcg-t+2}, if $(\tau_t(\mathcal{F}_1),\tau_t(\mathcal{F}_2))\neq (t+1,t+1)$, then
$|\mathcal{F}_1||\mathcal{F}_2|<a(k_1,t)a(k_2,t),$ contradicting the maximality  of $|\mathcal{F}_1||\mathcal{F}_2|$.
Hence $\tau_t(\mathcal{F}_1)=\tau_t(\mathcal{F}_2)=t+1$.
Let $\mathcal{T}_1$ and $\mathcal{T}_2$ be the collections of all $(t+1)$-dimensional $t$-covers of $\mathcal{F}_1$ and $\mathcal{F}_2$, respectively.

\textbf{Case~1.} Both $\mathcal{T}_1$ and $\mathcal{T}_2$ are $t$-intersecting.

Then $\mathcal{T}_1\cup \mathcal{T}_2$ is $t$-intersecting by Fact \ref{s-tt}, and $\tau_t(\mathcal{T}_1\cup\mathcal{T}_2)\in \{t, t+1\}$.
Suppose $\tau_t(\mathcal{T}_1\cup\mathcal{T}_2)=t$.
By Lemma~\ref{m}~(i), we obtain that $|\mathcal{T}_1|\leq \theta_{k_1-t+1}$ and $|\mathcal{T}_2|\leq \theta_{k_2-t+1}$.
If $(|\mathcal{T}_1|,|\mathcal{T}_2|)\neq\left(\theta_{k_1-t+1}, \theta_{k_2-t+1}\right)$, then $|\mathcal{F}_1||\mathcal{F}_2|<h(k_1,k_2,t)h(k_2,k_1,t)$ from Lemmas~\ref{not-contain}~and~\ref{<hh} (i), a contradiction.
Hence, $(|\mathcal{T}_1|,|\mathcal{T}_2|)=\left(\theta_{k_1-t+1}, \theta_{k_2-t+1}\right)$ and $\dim\left(\left(\sum_{A\in \mathcal{T}_1}A\right)\cap \left(\sum_{B\in \mathcal{T}_2}B\right)\right)\geq t+1$. When $\dim\left(\left(\sum_{A\in \mathcal{T}_1}A\right)\cap \left(\sum_{B\in \mathcal{T}_2}B\right)\right)=t+1$, we have $|\mathcal{F}_1||\mathcal{F}_2|<h(k_1,k_2,t)h(k_2,k_1,t)$ from Lemma~\ref{m}~(iib), a contradiction.
Then $\dim\left(\left(\sum_{A\in \mathcal{T}_1}A\right)\cap \left(\sum_{B\in \mathcal{T}_2}B\right)\right)\geq t+2$. Moreover, it follows from Lemma~\ref{m}~(iia) that
$$(\mathcal{F}_1,\mathcal{F}_2)= (\mathcal{H}(k_1,t;T,X,Y),\mathcal{H}(k_2,t;T,Y,X))$$
for some $X\in {V\brack k_1+1}$, $Y\in {V\brack k_2+1}$ and $T\in {V\brack t}$ such that $\dim(X\cap Y)\geq t+2$ and $T\subseteq X\cap Y$.

Suppose $\tau_t(\mathcal{T}_1\cup\mathcal{T}_2)=t+1$.
By Lemma~\ref{equivalent}, we know that $\max\{|\mathcal{T}_1|,|\mathcal{T}_2|\}\leq \theta_{t+2}$.
If $(|\mathcal{T}_1|,|\mathcal{T}_2|)\neq \left(\theta_{t+2},\theta_{t+2}\right)$, then Lemmas~\ref{not-contain}~and~\ref{<aa}~(i) imply that $|\mathcal{F}_1||\mathcal{F}_2|<a(k_1,t)a(k_2,t)$, a contradiction.
Hence, $|\mathcal{T}_1|=|\mathcal{T}_2|=\theta_{t+2}$. By Lemma~\ref{equivalent}, we have
$$(\mathcal{F}_1,\mathcal{F}_2)=\left(\mathcal{A}(k_1,t;Z), \mathcal{A}(k_2,t;Z)\right)$$
for some $Z\in {V\brack t+2}$.

\textbf{Case~2.} For $(i,j)\in \{(1,2),(2,1)\}$, $\mathcal{T}_i$ is not $t$-intersecting, and $|\mathcal{T}_j|=1$.

Suppose $\mathcal{T}_j\nsubseteq \mathcal{T}_i$. By Lemma~\ref{T=1}~(i), we have $|\mathcal{T}_i|\leq \theta_{2}\theta_{k_i-t+1}$. It follows from Lemmas~\ref{not-contain}, \ref{<hh}~(ii) and \ref{<cc}~(ii) that
\begin{align*}
	|\mathcal{F}_1||\mathcal{F}_2|
<\left\{
	\begin{array}{ll}
        h(k_1,k_2,t)h(k_2,k_1,t),&\text{if}~t=1, \vspace{0.1cm}\\
        c_1(k_i,t)c_2(k_j,k_i,t),&\text{if}~t\geq 2.
	\end{array}
	\right.
\end{align*}
This yields a contradiction.
Hence, $\mathcal{T}_j\subseteq \mathcal{T}_i$. Consequently, Lemma~\ref{T=1}~(ii) implies that $|\mathcal{T}_i|\leq q\theta_{t+1}\theta_{k_i-t}+1$. If $|\mathcal{T}_i|\leq q\theta_{t+1}\theta_{k_i-t}$, then Lemmas~\ref{not-contain} and \ref{<cc}~(i) give  $|\mathcal{F}_1||\mathcal{F}_2|<c_1(k_i,t)c_2(k_j,k_i,t)$, a contradiction.
Therefore, $|\mathcal{T}_i|=q\theta_{t+1}\theta_{k_i-t}+1$, and Lemma~\ref{T=1}~(ii) gives $$(\mathcal{F}_i,\mathcal{F}_j)= \left(\mathcal{C}_1(k_i,t;M,L), ~\mathcal{C}_2(k_j,t;M,L)\right)$$
for some $M\in{V\brack k_i+1}$ and $L\in{M\brack t+1}$.
If $(i,j)=(2,1)$, then $$|\mathcal{F}_1||\mathcal{F}_2|=c_1(k_2,t)c_2(k_1,k_2,t) <c_1(k_1,t)c_2(k_2,k_1,t)$$ from Lemma~\ref{<cc}~(iii), a contradiction. Hence, $(i,j)=(1,2)$ and so $$(\mathcal{F}_1,\mathcal{F}_2)= \left(\mathcal{C}_1(k_1,t;M,L), ~\mathcal{C}_2(k_2,t;M,L)\right).$$

\textbf{Case~3.} For $(i,j)\in\{(1,2),(2,1)\}$, $\mathcal{T}_i$ is not $t$-intersecting, and $\mathcal{T}_j$ is $t$-intersecting with $|\mathcal{T}_j|\geq 2$.

By Lemma \ref{T=2-C}, we have $|\mathcal{T}_j|=q+1$ and $|\mathcal{T}_i|\leq \max\left\{\theta_{k_i-t+1}+\theta_{t+2}-\theta_{2}, \theta_{k_i-t+2}\right\}$. Applying Lemma~\ref{not-contain}, if $|\mathcal{T}_i|\leq \theta_{k_i-t+2}$, then $|\mathcal{F}_1||\mathcal{F}_2|<h(k_1,k_2,t)h(k_2,k_1,t)$ from Lemma~\ref{<hh}~(iii);
if $|\mathcal{T}_i|\leq \theta_{k_i-t+1} +\theta_{t+2} -\theta_{2}$, then combining Lemmas~\ref{<hh}~(iv) and \ref{<aa}~(iv), we obtain
\begin{align*}
	|\mathcal{F}_1||\mathcal{F}_2|
<\left\{
	\begin{array}{ll}
        h(k_1,k_2,t)h(k_2,k_1,t),&\text{if}~k_i\geq 2t+1, \vspace{0.1cm}\\
        a(k_1,t)a(k_2,t),&\text{if}~t+2\leq k_i\leq 2t.
	\end{array}
	\right.
\end{align*}
Both estimates yield contradictions.

\textbf{Case~4.} Neither $\mathcal{T}_1$ nor $\mathcal{T}_2$ is $t$-intersecting.

By Lemma \ref{both-non-inter}, we have $|\mathcal{T}_1|=2$ and $|\mathcal{T}_2|\leq (q+1)^2$, $|\mathcal{T}_1|\leq (q+1)^2$ and $|\mathcal{T}_2|=2$, or $|\mathcal{T}_1|\leq 2q+1$ and $|\mathcal{T}_2|\leq 2q+1$.
It follows from Lemmas \ref{not-contain}, \ref{<aa}~(ii) and \ref{<aa}~(iii) that $|\mathcal{F}_1||\mathcal{F}_2|<a(k_1,t)a(k_2,t)$,
a contradiction.

This completes the proof of Theorem \ref{main-1}.$\qed$

\section{Some inequalities}\label{3}

In this section, we prove some inequalities used in the proof of Theorem \ref{main-1}. We always assume that $n$, $k$, $\ell$ and $t$ are positive integers with $\min\{k,\ell\}\geq t+2$ and $n\geq 2\cdot\max\{k,\ell\}+k+\ell-t+5$.

Let $g(m,x,y,t)$, $a(x,t)$, $h(x,y,t)$, $c_1(x,y,t)$ and $c_2(y,t)$ be as in \eqref{equ-3}--\eqref{equ-12}. Set
\begin{align}
\widetilde{g}(m,x,y,t)&=g(m,x,y,t) {n-t-1\brack x-t-1}^{-1},\label{equ-6}\\
\widetilde{a}(x,t)&=a(x,t){n-t-1\brack x-t-1}^{-1},\label{equ-4}\\
\widetilde{h}(x,y,t)&=h(x,y,t){n-t-1\brack x-t-1}^{-1},\label{equ-5}\\
\widetilde{c}_1(x,y,t)&=c_1(x,y,t){n-t-1\brack x-t-1}^{-1},\label{equ-13}\\
\widetilde{c}_2(y,t)&=c_2(y,t){n-t-1\brack y-t-1}^{-1}.\label{equ-14}
\end{align}
Note that
\begin{align}\label{equ-6_m+}
m<\widetilde{g}(m,x,y,t)=m+\theta_{t+1}\theta_{y-t+1}^2\cdot \frac{q^{x-t-1}-1}{q^{n-t-1}-1}<m+q^{-n+(x+2y-t+3)}.
\end{align}

Let $x$ and $b$ be positive integers with $x\geq t+b$. For $j\geq 0$, we have
\begin{align}\label{bbp}
q^{(x-t-b)j}{n-t-b-j\brack x-t-b}{n-t-b\brack x-t-b}^{-1}
=\prod_{i=0}^{x-t-b-1}\frac{q^{n-t-b-i}-q^j}{q^{n-t-b-i}-1}\leq 1.
\end{align}

\begin{lem}\label{<hh}
The following statements hold.
 \begin{enumerate}[{\rm(i)}]
   \item $g(\theta_{\ell-t+1}-1,k,\ell,t) g(\theta_{k-t+1},\ell,k,t)<h(k,\ell,t)h(\ell,k,t).$
   \item $g(\theta_{2}\theta_{\ell},k,\ell,1) g(1,\ell,k,1) <h(k,\ell,1)h(\ell,k,1).$
   \item $g(\theta_{\ell-t+2},k,\ell,t) g\left(q+1,\ell,k,t\right) <h(k,\ell,t)h(\ell,k,t).$
   \item If $\ell\geq 2t+1$, then $g(\theta_{\ell-t+1}+\theta_{t+2} -\theta_{2},k,\ell,t)g(q+1,\ell,k,t) <h(k,\ell,t)h(\ell,k,t).$
 \end{enumerate}
\end{lem}
\proof
Let $x$ and $y$ be positive integers with $\min\{x,y\}\geq t+2$.
By using Lemma \ref{q<q} repeatedly, we have
$$h(x,y,t)=\sum_{i=1}^{y-t+1} q^{(i-1)(x-t)}{n-t-i\brack x-t-1}+q^{x-t+1}\theta_{t},$$
and together with \eqref{bbp}, we obtain
\begin{align}\label{h<qy}
\widetilde{h}(x,y,t)<\sum_{i=1}^{y-t+1}q^{(i-1)(x-t)} {n-t-i\brack x-t-1}{n-t-1\brack x-t-1}^{-1}+1\leq \sum_{i=1}^{y-t+1}q^{i-1}+1\leq q^{y-t+1}.
\end{align}
Set
\begin{align*}
v(x,y,t)&=\left(\theta_{t+1}\theta_{y-t+1}^2 {n-t-2\brack x-t-2} -q^{x-t+1}\theta_{t}\right) \cdot {n-t-1\brack x-t-1}^{-1},\\
u(x,y,t)&=\left(\theta_{y-t+1}{n-t-1\brack x-t-1}-\sum_{i=1}^{y-t+1} q^{(i-1)(x-t)}{n-t-i\brack x-t-1}\right)\cdot {n-t-1\brack x-t-1}^{-1}.
\end{align*}
Then we have $\widetilde{g}(\theta_{y-t+1},x,y,t) -\widetilde{h}(x,y,t)=u(x,y,t)+v(x,y,t)$.

By Lemma~\ref{q<q} and \eqref{bbp}, we obtain
\begin{align}\label{u<q^n-2k-2l}
u(x,y,t)&=\left(\sum_{i=2}^{y-t+1}q^{i-1}\left(\sum_{j=0}^{i-2} q^{j(x-t-1)}\frac{{n-t-2-j\brack x-t-2}}{{n-t-2\brack x-t-2}}\right)\right)\cdot \frac{q^{x-t-1}-1}{q^{n-t-1}-1}\notag \\
&<\left(\sum_{i=2}^{y-t+1}q^{i-1}\left(\sum_{j=0}^{i-2} q^{j}\right)\right)\cdot q^{x-n}
<q^{-n+x+2y-2t+1}.
\end{align}
Moreover, Lemma~\ref{q<q} gives $v(x,y,t)<q^{-n+x+2y-t+3}$, and so
\begin{align}\label{v<x+2y}
\widetilde{g}(\theta_{y-t+1},x,y,t)<\widetilde{h}(x,y,t) +q^{-n+x+2y-t+4}.
\end{align}

(i) From \eqref{equ-6} and \eqref{equ-5}, it suffices to show that
$$ \widetilde{g}(\theta_{\ell-t+1}-1,k,\ell,t) \widetilde{g}(\theta_{k-t+1},\ell,k,t) <\widetilde{h}(k,\ell,t) \widetilde{h}(\ell,k,t).$$
It follows from \eqref{equ-6_m+}, \eqref{h<qy} and \eqref{v<x+2y} that
\begin{align*}
&~~~~\widetilde{h}(k,\ell,t) \widetilde{h}(\ell,k,t)- \widetilde{g}(\theta_{\ell-t+1}-1,k,\ell,t) \widetilde{g}(\theta_{k-t+1},\ell,k,t)\\
&=\widetilde{h}(k,\ell,t) \widetilde{h}(\ell,k,t)- \widetilde{g}(\theta_{\ell-t+1},k,\ell,t) \widetilde{g}(\theta_{k-t+1},\ell,k,t) +\widetilde{g}(\theta_{k-t+1},\ell,k,t)\\
&> -q^{-n+2k+\ell-t+4}\widetilde{h}(k,\ell,t) -q^{-n+k+2\ell-t+4}\widetilde{h}(\ell,k,t) -q^{-2n+3k+2\ell-2t+8}+\widetilde{g}(\theta_{k-t+1},\ell,k,t) \\
&>-2q^{-n+2k+2\ell-2t+5}-q^{-2n+3k+3\ell-2t+8} +\theta_{k-t+1}>0.
\end{align*}
and so (i) holds.

(ii) From \eqref{u<q^n-2k-2l} and \eqref{h<qy}, it follows that
\begin{align*}
\widetilde{h}(k,\ell,1)\widetilde{h}(\ell,k,1)
&= \left(\widetilde{h}(k,\ell,1) -\theta_{\ell}\right)\widetilde{h}(\ell,k,1) +\left(\widetilde{h}(\ell,k,1)-\theta_{k}\right)\theta_{\ell} +\theta_k\theta_\ell\\
&>-u(k,\ell,1)\widetilde{h}(\ell,k,1) -u(\ell,k,1)\theta_{\ell}+\theta_k\theta_\ell\\
&>-2q^{-n+2k+2\ell-1}+\theta_k\theta_\ell \geq \theta_k\theta_\ell-1.
\end{align*}
By Lemma~\ref{q<q}, we have
$$\widetilde{g}(\theta_{2}\theta_{\ell},k,\ell,1) \widetilde{g}(1,\ell,k,1)<\left(\theta_{2}\theta_{\ell}+q^{-n+k+2\ell+2}\right) \left(1+q^{-n+2k+\ell+2}\right)<\theta_{2}\theta_{\ell}+2.$$
Hence, (ii) follows from \eqref{equ-6} and \eqref{equ-5}.

(iii) It is clear that
\begin{align}
\theta_{\ell-t+2}&=\theta_{\ell-t+1}+q^{\ell-t+1} =q\theta_{\ell-t}+q^{\ell-t+1}+1, \label{l+2=l+1+q}\\
\theta_{k-t+1}&=q^2\theta_{k-t-1}+q+1.\label{k+1=k-1+q}
\end{align}
From \eqref{equ-6_m+}, \eqref{l+2=l+1+q} and \eqref{k+1=k-1+q}, it follows  that
\begin{align*}
&~~~~\widetilde{g}(\theta_{\ell-t+2},k,\ell,t) \widetilde{g}(q+1,\ell,k,t)\notag \\
&=\left(\widetilde{g}(\theta_{\ell-t+1}-1,k,\ell,t) +q^{\ell-t+1}+1 \right) \left(\widetilde{g}(\theta_{k-t+1},\ell,k,t) -q^2\theta_{k-t-1} \right)\\
&=\widetilde{g}(\theta_{\ell-t+1}-1,k,\ell,t) \widetilde{g}(\theta_{k-t+1},\ell,k,t)-Q,
\end{align*}
where
$$Q=q^2\theta_{k-t-1}\cdot \widetilde{g}(\theta_{\ell-t+2},k,\ell,t) -\left(q^{\ell-t+1}+1\right) \cdot\widetilde{g}(\theta_{k-t+1},\ell,k,t) +q^2\theta_{k-t-1}(q^{\ell-t+1}+1).$$
Then, by Lemma \ref{<hh}~(i), it suffices to show that
$Q>0$.

Suppose $(k,\ell)\neq (t+2,t+2)$.
By \eqref{equ-6_m+}, we have
\begin{align*}
Q&>q^2\theta_{k-t-1}\cdot \theta_{\ell-t+2} -\left(q^{\ell-t+1}+1\right) \left(\theta_{k-t+1}+q^{-n+2k+\ell-t+3}\right)\\
&>q^2\theta_{k-t-1}\theta_{\ell-t+2}-
\left(q^{\ell-t+1}+1\right)\cdot\theta_{k-t+1}-q^{-t}\\
&=q^3\theta_{k-t-1}\theta_{\ell-t} -(q+1)\left(q^{\ell-t+1}+1\right)-q^{-t},
\end{align*}
where equality holds due to \eqref{l+2=l+1+q} and \eqref{k+1=k-1+q}.
If $k=t+2$, then $\theta_{k-t-1}=1$ and $\ell\geq t+3$ from $(k,\ell)\ne(t+2,t+2)$. Thus
\begin{align*}
q^3\theta_{k-t-1}\theta_{\ell-t} -(q+1)\left(q^{\ell-t+1}+1\right)
\geq q^3-q-1>q.
\end{align*}
If $k\ge t+3$, then $\theta_{k-t-1}\geq \theta_{2}=q+1$.
Hence,
\begin{align*}
q^3\theta_{k-t-1}\theta_{\ell-t} -(q+1)\left(q^{\ell-t+1}+1\right)
\geq
(q+1)\left(q^{\ell-t+1}(q-1)-1\right)>q.
\end{align*}
Thus, $Q>0$.

Suppose $k=\ell=t+2$. Then
\begin{align*}
Q&=q^2\cdot \widetilde{g}(\theta_{4},t+2,t+2,t) -\left(q^{3}+1\right) \cdot\widetilde{g}(\theta_{3},t+2,t+2,t) +q^2(q^{3}+1)\\
&>q^2\theta_{4}-\left(q^{3}+1\right) \left(\theta_{3}+q^{-n+2t+9}\right)+q^2(q^{3}+1)\\
&>q^2\theta_{4}-\left(q^{3}+1\right) \left(\theta_{3}+q^{-1}\right)+q^5+q^2>0.
\end{align*}

Therefore, (iii) holds.

(iv) By $\ell\geq 2t+1$, we have
$\theta_{\ell-t+2}>\theta_{\ell-t+1}+\theta_{t+2}-\theta_{2}.$
Then $g\left(\theta_{\ell-t+1}+\theta_{t+2} -\theta_{2},k,\ell,t\right) <g\left(\theta_{\ell-t+2},k,\ell,t\right)$. Thus (iv) follows from (iii).$\qed$

\begin{lem}\label{<aa}
The following statements hold.
\begin{enumerate}[{\rm(i)}]
  \item $g(\theta_{t+2}-1,k,\ell,t) g(\theta_{t+2},\ell,k,t)<a(k,t)a(\ell,t).$
  \item $g(2,k,\ell,t) g((q+1)^2,\ell,k,t)<a(k,t)a(\ell,t).$
  \item $g(2q+1,k,\ell,t) g(2q+1,\ell,k,t)<a(k,t)a(\ell,t).$
  \item If $\ell\leq 2t$, then $g(\theta_{\ell-t+1}+\theta_{t+2} -\theta_{2},k,\ell,t)g(q+1,\ell,k,t) <a(k,t)a(\ell,t).$
\end{enumerate}
\end{lem}
\proof For positive integers $m$ and $x$ with $x\geq t+1$, by Lemma \ref{q<q}, we have
\begin{align*}
a(x,t)=\theta_{t+2}{n-t-1\brack x-t-1}-q\theta_{t+1}{n-t-2\brack x-t-2},
\end{align*}
and so
\begin{align}\label{aa>t+2+}
\widetilde{a}(x,t)> \theta_{t+2}-q^{-n+x+t+2}.
\end{align}
From \eqref{equ-6} and \eqref{equ-4}, we need to proof that
$$\widetilde{a}(k,t) \widetilde{a}(\ell,t)> \widetilde{g}(m_1,k,\ell,t) \widetilde{g}(m_2,\ell,k,t),$$
where $(m_1,m_2)$ is, respectively, $\left(\theta_{t+2}-1,~\theta_{t+2}\right)$, $\left(2,~(q+1)^2\right)$, $\left(2q+1,~2q+1\right)$ and \\ $\left(\theta_{\ell-t+1}+\theta_{t+2}-\theta_{2},~q+1\right)$.
Since $n\geq 2\cdot\max\{k,\ell\}+k+\ell-t+5$, it is clear that $\widetilde{a}(k,t) \widetilde{a}(\ell,t)>\left(\theta_{t+2} -q^{-5}\right)^2$.
Combing with \eqref{equ-6_m+}, it suffices to show that
$$P(m_1,m_2):=\left(\theta_{t+2} -q^{-5}\right)^2-\left(m_1+q^{-2}\right) \left(m_2+q^{-2}\right)>0.$$

(i) We have $P(\theta_{t+2}-1,\theta_{t+2}) >\left(1-2q^{-2}-2q^{-5}\right)\cdot\theta_{t+2}>0,$
and (i) holds.

(ii) We have $P(2,(q+1)^2) >q(q+1)\theta_{t+2} -2(q+1)^2 -4q\geq 2q^2(q+1)-4q
>0,$
and (ii) holds.

(iii) Since
$2q+1<\theta_{t+2}-1$, we have
$$\widetilde{g}(2q+1,k,\ell,t)\widetilde{g}(2q+1,\ell,k,t) <\widetilde{g}(\theta_{t+2}-1,k,\ell,t) \widetilde{g}(\theta_{t+2},\ell,k,t).$$
It follows from (i) that (iii) holds.

(iv) Since $\ell\leq 2t$, we have
\begin{align*}
P(\theta_{\ell-t+1}+\theta_{t+2} -\theta_{2},q+1)&\geq P(\theta_{t+1}+\theta_{t+2} -\theta_{2},\theta_{2})\\
&> \theta_{t+2}^2-2q^{-5}\theta_{t+2} -(\theta_{t+1}+\theta_{t+2}+q^{-2})(\theta_{2}+q^{-2})\\
&>\theta_{t+2}^2-2q^{-5}\theta_{t+2} -2\theta_{t+2}(\theta_{2}+q^{-2})\\
&\geq\theta_{t+2}(\theta_{3}-2q^{-5} -2\theta_{2}-2q^{-2})>0,
\end{align*}
and so (iv) holds.$\qed$

\begin{lem}\label{<cc}
The following statements hold.
\begin{enumerate}[\rm(i)]
  \item  $g(q\theta_{t+1}\theta_{\ell-t},k,\ell,t) g(1,\ell,k,t) <c_1(\ell,t)c_2(k,\ell,t).$
  \item If $t\geq 2$, then $g(\theta_{2}\theta_{\ell-t+1},k,\ell,t) g(1,\ell,k,t) <c_1(\ell,t)c_2(k,\ell,t).$
  \item If $\ell\geq k+1$, then $c_1(k,t)c_2(\ell,k,t)<c_1(\ell,t)c_2(k,\ell,t)$.
\end{enumerate}
\end{lem}
\proof Let $x$ and $y$ be positive integers with $\min\{x,y\}\geq t+2$. Then, from Lemma~\ref{q<q},
\begin{align}\label{c<ell+1}
c_2(x,y,t)= \theta_{t+1} \sum_{i=1}^{y-t}q^{i(x-t)}{n-t-1-i\brack x-t-1}+{n-t-1\brack x-t-1}.
\end{align}

(i) From $q\theta_{\ell-t}=\sum_{i=1}^{\ell-t}q^i$, we have
\begin{align*}
\sum_{i=1}^{\ell-t}q^{i(k-t)}{n-t-1-i\brack k-t-1} -q\theta_{\ell-t}{n-t-1\brack k-t-1}&=-\sum_{i=1}^{\ell-t}q^i\left(\sum_{j=0}^{i-1} q^{j(k-t-1)}{n-t-2-j\brack k-t-2}\right),
\end{align*}
and together with \eqref{bbp} and \eqref{c<ell+1}, we obtain
\begin{align*}
&~~~~\widetilde{c}_2(k,\ell,t)- \widetilde{g}(q\theta_{t+1}\theta_{\ell-t},k,\ell,t)\\
&=1-\theta_{t+1} \sum_{i=1}^{\ell-t} q^i\left(\sum_{j=0}^{i-1} q^{j(k-t-1)} \frac{{n-t-2-j\brack k-t-2}}{{n-t-1\brack k-t-1}}\right)-\theta_{t+1}\theta_{\ell-t+1}^2\cdot \frac{q^{k-t-1}-1}{q^{n-t-1}-1}\\
&>1-q^{t+1} \sum_{i=1}^{\ell-t} q^i\left(\sum_{j=0}^{i-1} q^{j-(n-k)}\right)-q^{-n+k+2\ell-t+3}\\
&>1-q^{-n+k+2\ell-t+2} -q^{-n+k+2\ell-t+3}.
\end{align*}
Hence,
\begin{align*}
&~~~~\widetilde{c}_1(\ell,t)\widetilde{c}_2(k,\ell,t)- \widetilde{g}(q\theta_{t+1}\theta_{\ell-t},k,\ell,t) \widetilde{g}(1,\ell,k,t)\\
&>\widetilde{c}_2(k,\ell,t) -\widetilde{g}(q\theta_{t+1}\theta_{\ell-t},k,\ell,t) \left(1+\theta_{t+1}\theta_{k-t+1}^2 \cdot\frac{q^{\ell-t-1}-1}{q^{n-t-1}-1}\right)\\
&>1-q^{-n+k+2\ell-t+2} -q^{-n+k+2\ell-t+3}- q^{-n+2 k+\ell-t+3}\cdot \widetilde{g}(q\theta_{t+1}\theta_{\ell-t},k,\ell,t)\\
&>1-2q^{-5}- q^{-1}-q^{-13}>0.
\end{align*}
Then (i) holds.

(ii) By $t\geq 2$, we have
$q\theta_{t+1}\theta_{\ell-t}-\theta_{2}\theta_{\ell-t+1} = q\theta_{\ell-t}\left(\theta_{t+1}-\theta_{2}\right) -\theta_{2}>0.$
Therefore,
\begin{align*}
g(\theta_{2}\theta_{\ell-t+1},k,\ell,t) g(1,\ell,k,t) <g(q\theta_{t+1}\theta_{\ell-t},k,\ell,t) g(1,\ell,k,t),
\end{align*}
and so (ii) follows from (i).

(iii) By \eqref{equ-13} and \eqref{equ-14}, we have $\widetilde{c}_1(\ell,t)>1$, $\widetilde{c}_1(k,t)=1+q^{k-t}\theta_{t+1}{n-t-1\brack k-t-1}^{-1}$, and
\begin{align*}
\widetilde{c}_2(k,\ell,t)\geq \theta_{t+1} \sum_{i=1}^{k-t}
q^{i(k-t)}\frac{{n-t-1-i\brack k-t-1}}{{n-t-1\brack k-t-1}}+\theta_{t+1}q^{(k-t+1)(k-t)}\frac{{n-k-2\brack k-t-1}}{{n-t-1\brack k-t-1}}
\end{align*}
from \eqref{c<ell+1} and $\ell\geq k+1$. Then
\begin{align*}
\widetilde{c}_1(\ell,t)\widetilde{c}_2(k,\ell,t)- \widetilde{c}_1(k,t)\widetilde{c}_2(\ell,k,t) &>\widetilde{c}_2(k,\ell,t)- \widetilde{c}_2(\ell,k,t)\left(1+\frac{q^{k-t}\theta_{t+1}} {{n-t-1\brack k-t-1}}\right)\\
&>\theta_{t+1}\sum_{i=1}^{k-t}Q_i\notag +\frac{q^{k-t}\theta_{t+1}}{{n-t-1\brack k-t-1}}\cdot R,
\end{align*}
where
\begin{align*}
Q_i=q^{i(k-t)}\frac{{n-t-1-i\brack k-t-1}}{{n-t-1\brack k-t-1}}- q^{i(\ell-t)}\frac{{n-t-1-i\brack \ell-t-1}}{{n-t-1\brack \ell-t-1}},~~~~R=q^{(k-t)^2}{n-k-2\brack k-t-1}-\widetilde{c}_2(\ell,k,t).
\end{align*}
To prove (iii), it suffices to show that $R>0$ and $Q_i>0$ for $0\leq i\leq k-t-1$.

By $\ell\geq k+1$, we obtain
\begin{align*}
Q_i=q^{i}\prod_{j=0}^{k-t-2} \frac{q^{n-t-1-j}-q^{i}}{q^{n-t-1-j}-1} \left(1-\prod_{j=k-t-1}^{\ell-t-2} \frac{q^{n-t-1-j}-q^{i}}{q^{n-t-1-j}-1}\right)>0.
\end{align*}
From Lemma~\ref{q<q}, \eqref{c<ell+1} and \eqref{bbp}, it follows that
\begin{align*}
R&>q^{(k-t-1)(n-2k+t-1)+(k-t)^2}-\theta_{t+1} \sum_{i=1}^{k-t}q^{i}-1\\
&> q^{(2k+7)(k-t-1)+(k-t)^2}-q^{k+2}-1>0,
\end{align*}
due to $n\geq k+3\ell-t+5$.

Therefore,  (iii) follows.
$\qed$

\section*{Acknowledgments}
B. Lv is supported by the National Natural Science Foundation of China (12571347 \& 12131011), and Beijing Natural Science Foundation (1252010).  K. Wang is supported by the National Natural Science Foundation of China (12131011 \& 12571347) and Beijing Natural Science Foundation (1252010 \& 1262010).

\end{document}